\documentclass{amsart}
\usepackage{amsfonts,amsmath,amsthm,amssymb}
\usepackage[linesnumbered,ruled,noend]{algorithm2e}
\usepackage{comment}
\usepackage{float}
\usepackage{latexsym}
\usepackage{graphicx}
\usepackage{multicol}
\usepackage{color}
\usepackage{pifont}
\usepackage{enumerate}
\usepackage{lmodern}
\usepackage{makecell}
\usepackage{adjustbox}
\usepackage[backref,pdftex=true,colorlinks,linkcolor=blue,citecolor=red,pdfstartview=FitH]{hyperref}
\usepackage{setspace}
\usepackage[all]{xy}
\usepackage{pb-diagram,pb-xy}
\usepackage{pdflscape}
\usepackage{chngpage}
\usepackage{breakurl}

\usepackage{tabu}

\makeatletter
\def\BState{\State\hskip-\ALG@thistlm}
\makeatother

\theoremstyle{plain}
\newtheorem{theorem}{Theorem}[section]
\newtheorem{lemma}[theorem]{Lemma}

\theoremstyle{definition} 

\newtheorem*{example*}{Example}

\DeclareMathOperator{\Rad}{Rad}
\newcommand{\F}{\mathbb{F}}
\newcommand{\Q}{\mathbb{Q}}
\newcommand{\Z}{\mathbb{Z}}

\newcommand{\p}{\mathfrak{p}}
\newcommand{\q}{\mathfrak{q}}

\newcommand{\OO}{\mathcal{O}}

\newcommand{\cP}{\mathcal{P}}
\newcommand{\cE}{\mathcal{E}}

\DeclareMathOperator{\ord}{ord}

\newcommand{\fq}{\mathfrak{q}}
\newcommand{\sS}{\mathfrak{S}}

\begin{document}

\title[]{On perfect powers that are sums of cubes of a five term arithmetic progression}

\author{Alejandro Arg\'{a}ez-Garc\'{i}a}
\address{Facultad de Ingenier\'{i}a Qu\'{i}mica, Universidad Aut\'{o}noma de Yucat\'{a}n. Perif\'{e}rico Norte Kil\'{o}metro 33.5, Tablaje Catastral 13615 Chuburna de Hidalgo Inn, M\'{e}rida, Yucat\'{a}n, M\'{e}xico. C.P. 97200 ;\newline Centro de Invetigaci\'{o}n en Matem\'{a}ticas, A.C. - Unidad M\'{e}rida. Parque Cient\'{i}fico y Tecnol\'{o}gico de Yucat\'{a}n Km 5.5  Carretera Sierra Papacal - Chuburn\'{a} Puerto Sierra Papacal, M\'{e}rida, Yucat\'{a}n, M\'{e}xico. C.P. 97302}
\email{alejandroargaezg@gmail.com/\newline alejandro.argaez@correo.uady.mx/\newline alejandro.argaez@cimat.mx}

\thanks{
The author was supported by the Research Grant 221183 FOMIX, CONACYT - Government of Yucat\'an.
}

\date{\today}


\begin{abstract}
We prove that the equation $(x-2r)^3 + (x-r)^3 + x^3 + (x+r)^3 + (x+2r)^3= y^p$ only has solutions which satisfy $xy=0$ for $1\leq r\leq 10^6$ and $p\geq 5$ prime.

\end{abstract}
\maketitle

\section{Introduction}

Finding perfect powers that are sums of terms in an arithmetic progression has received much interest; recent contributions can be found in 
\cite{ArgaezPatel}, \cite{BGP}, \cite{BPS1}, \cite{BPS2}, \cite{BPSS} 
\cite{KoutsianasPatel},
\cite{KunduPatel},
\cite{Patelsquares}, \cite{PatelSiksek}, 
\cite{P}, \cite{P2}, \cite{Soydan}
\cite{ZB}, \cite{Zhang} and \cite{Zhang2}.

We are interested in solutions $(x,y,r)$ where $x$, $y$ and $r$ are coprime.

In this paper, we prove the following: 
\begin{theorem}\label{thm:principal01}
Let $p \ge 5$ be a prime. The equation
\begin{equation}\label{eq:main}
(x-2r)^3 + (x-r)^3 + x^3 + (x+r)^3 + (x+2r)^3= y^p  \quad x,r,y,p \in \Z, \; \gcd(x,r)=1,
\end{equation}
with $0 < r\le 10^6$ only has solutions which satisfy $xy=0$.
\end{theorem}

The restriction $\gcd(x,r)=1$ is natural one, for 
otherwise it is easy to construct artificial solutions by
scaling. 
We use a combination of techniques in the resolution of \eqref{eq:main}, the main ones being; a result of Mignotte based on linear form in logarithms (\cite[Chapter~12, p.~423]{Cohen2}); the method of Chabauty (\cite{SiksekChab,Stoll, MP}), the theorem due to Bilu, Hanrot and Voutier on primitive divisors (\cite{BHV}), as well as some various elementary techniques.

\section{Background}\label{sec:background}

Here, we record some essential Theorems and Lemmas which are necessary for the computaions in Section~\ref{sec:tablecomps}.

\begin{theorem}\label{thm:mignotte}(Mignotte)
Assume that the exponential Diophantine inequality 
\[
\mid ax^n-by^n\mid\leq c, \qquad \text{with } a,b,c\in\Z_{\geq 0} \text{ and } a\neq b
\]
has a solution in strictly positive integers $x$ and $y$ with $\max\{x,y\}>1$. Let $A=\max\{a,b,3\}$. Then 
\[
n\leq \max\left\lbrace3\log(1.5\mid c/b\mid),\dfrac{7400\log A}{\log\left(1+\log A/\log(\mid a/b\mid) \right)} \right\rbrace.
\]
\end{theorem}

\subsection{Criteria for eliminating equations of signature $(p,2p,2)$} \label{sec:criteria}

We first apply a descent to equation~\eqref{eq:main} in Section~\ref{casen2}. We are left with equations of the form:
\begin{equation}\label{eqn:form}
aw_2^{p} - bw_1^{2p} = cr^2
\end{equation}
where $p$ is an odd prime and $a,b,c$ are positive integers
satisfying $\gcd(a,b,c)=1$.

The criteria to follow in order to determine whether $(\ref{eqn:form})$ has solutions was previously presented in \cite{Patelthesis, BPS2,ArgaezPatel}. We start with the following lemma that give us a criteria for the nonexistence of solutions.

\begin{lemma}\label{lem:Sophiecriterion}
Let $p \ge 3$ be a prime. Let $a$, $b$ and $c$ be positive integers such that 
$\gcd(a,b,c)=1$. Let $q=2k p+1$ be a prime that does
not divide $a$. Define
\begin{equation}\label{eqn:mu}
\mu(p,q)=\{ \eta^{2p} \; : \; \eta \in \F_q \}
=\{0\} \cup \{ \zeta \in \F_q^* \; : \; \zeta^{k}=1\}
\end{equation}
and
$$
B(p,q)=\left\{ \zeta \in \mu(p,q) \; : \; ((b \zeta+c)/a)^{2k} \in \{0,1\} \right\} \, .
$$
If $B(p,q)=\emptyset$, then equation~\eqref{eqn:form} does not have integral solutions.
\end{lemma}

\subsection{Local Solubility}\label{sub:locsol}

After reducing the number of 
equations using Lemma \ref{lem:Sophiecriterion}, we give the next step and use classical local solubility methods to conclude nonexistence of solutions for many tuples $(a,b,c,p)$ in equation \eqref{eqn:form}. 

These methods work as follows: let $g=\Rad(\gcd(a,c))$ and suppose that $g > 1$. Recall the condition $\gcd(a,b,c)=1$.  Then $g \mid w_1$, and we can write $w_1=g w_1^\prime$. Thus
\[
a w_2^p- b g^{2 p} {w_1^\prime}^{2p}=c. 
\]
Removing a factor of $g$ from the coefficients, we obtain
\[
a^\prime {w_2}^p - b^\prime {w_1^\prime}^{2p}=c^\prime,
\]
where $a^\prime=a/c$ and $c^\prime=c/g<c$. Similarly, if $h=\gcd(b,c)>1$, we obtain
\[
a^\prime {w_2^\prime}^p - b^\prime {w_1}^{2p}=c^\prime,
\]
where $c^\prime=c/h<c$. Applying these operations repeatedly, we arrive at an equation of the form
\begin{equation}\label{eqn:predescent}
A \rho^p-B \sigma^{2p}=C
\end{equation}
where $A$, $B$, $C$ are now pairwise coprime. A necessary condition for the existence of solutions is that for any odd prime $q \mid A$, the residue $-BC$ modulo $q$ is a square.
Besides this basic test, we also check for local solubility at the primes dividing $A$, $B$, $C$, and all primes $q \le 19$.

\subsection{A Descent}\label{sub:furtherdesc}

If local techniques previously presented failed to rule out solutions 
to equation \eqref{eqn:form} for particular coefficients and exponent $(a,b,c,p)$ 
then we may perform a further descent to rule out solutions. 
With $A$, $B$, $C$ as in \eqref{eqn:predescent}
we let
\[
B^\prime=\prod_{\text{$\ord_q(B)$ is odd}} q.
\]
Thus $B B^\prime=v^2$. Write $A B^\prime=u$ and $C B^\prime=m n^2$
with $m$ squarefree. Rewrite \eqref{eqn:predescent}
as
\[
(v \sigma^p+n \sqrt{-m})(v \sigma^p-n \sqrt{-m})=u \rho^p. 
\]
Let $K=\Q(\sqrt{-m})$ and $\OO$ be its ring of integers. Let $\sS$ contain the prime ideals of $\OO$ that divide $u$ or $2n \sqrt{-m}$. Clearly 
$(v\sigma^p+n \sqrt{-m}) {K^*}^p$ belongs to the ``$p$-Selmer group''
\[
K(\sS,p)=\{\epsilon \in K^*/{K^*}^p \; : \; 
\text{$\ord_\cP(\epsilon) \equiv 0 \mod{p}$ for all $\cP \notin \sS$}
\}.
\]
This is an $\F_p$-vector space of finite dimension can be computed by \texttt{Magma} using the command \texttt{pSelmerGroup}. Let
\[
\cE=\{ \epsilon \in K(\sS,p) \; : \; Norm(\epsilon)/u \in {\Q^*}^p \}.
\]
It follows that
\begin{equation}\label{eqn:furtherdescent}
v \sigma^p+n \sqrt{-m}=\epsilon \eta^p,
\end{equation}
where $\eta \in K^*$ and $\epsilon \in \cE$.

We end up with the last criteria.

\begin{lemma}\label{lem:valuative}
Let $\fq$ be a prime ideal of $K$. Suppose one of the following
holds:
\begin{enumerate}[$(i)$]
\item $\ord_\fq(v)$, $\ord_\fq(n\sqrt{-m})$, $\ord_\fq(\epsilon)$
are pairwise distinct modulo $p$;
\item $\ord_\fq(2v)$, 
$\ord_\fq(\epsilon)$, $\ord_\fq(\overline{\epsilon})$
are pairwise distinct modulo $p$;
\item $\ord_\fq(2 n \sqrt{-m})$, 
$\ord_\fq(\epsilon)$, $\ord_\fq(\overline{\epsilon})$
are pairwise distinct modulo $p$.
\end{enumerate}
Then there is no $\sigma \in \Z$ and $\eta \in K$ satisfying \eqref{eqn:furtherdescent}.
\end{lemma}

\begin{lemma}\label{lem:furtherdescent}
Let $q=2k p+1$ be a prime. 
Suppose $q\OO=\fq_1 \fq_2$ where $\fq_1$, $\fq_2$
are distinct, and such that $\ord_{\fq_j}(\epsilon)=0$
for $j=1$, $2$. Let 
\[
\chi(p,q)=\{ \eta^p \; : \; \eta \in \F_q \}.
\]
Let
\[
C(p,q)=\{\zeta \in \chi(p,q) \; : \;
((v \zeta+n\sqrt{-m})/\epsilon)^{2k} \equiv \text{$0$ or $1 \mod{\fq_j}$
for $j=1$, $2$}\}.
\]
Suppose $C(p,q)=\emptyset$. 
Then there is no $\sigma \in \Z$
and $\eta \in K$ satisfying \eqref{eqn:furtherdescent}.
\end{lemma}

\subsection{Thue equations}
Finally for the remaining equations that couldn't be ruled out, they can be considered Thue equations by letting $\sigma = w_2$ and $\tau = w_1^2$
\[
a\sigma^p - b \tau^p = c 
\]
where the exponent is a prime $p$. We use \texttt{Magma's} Thue solver \cite{magma} and \texttt{PARI/GP's} \textit{thueinit}, \textit{thue} commands as the final test to determine whether the equations has solutions.

\subsection{Prime divisors of Lehmer sequences}

A \textit{Lehmer pair} is a pair $\alpha$, $\beta$ of algebraic integers such that $(\alpha+\beta)^2$ and $\alpha\beta$ are nonzero coprime rational integers and $\alpha/\beta$ is not a root of unity. The \textit{Lehmer sequence} associated to the Lehmer pair $(\alpha,\beta)$ is
\begin{align}\label{eqn:lehmersequence}
\tilde{u}_n=\tilde{u}_n(\alpha,\beta)=\begin{cases}
\dfrac{\alpha^n-\beta^n}{\alpha-\beta} & n~\text{is odd}\\
\dfrac{\alpha^n-\beta^n}{\alpha^2-\beta^2} & n~\text{is even}
\end{cases}
\end{align}

A prime $p$ is called a \textit{primitive divisor} of $\tilde{u}_n$ if it divides $\tilde{u}_n$ but does not divide $(\alpha^2-\beta^2)\cdot\tilde{u}_1\cdots\tilde{u}_{n-1}$. We now state the following celebrated theorem due to Bilu, Hanrot and Voutier \cite{BHV}.

\begin{theorem}\label{BHV}
Let $\alpha$, $\beta$ be a Lehmer pair. Then $\tilde{u}_n(\alpha,\beta)$ has a primitive root for all $n>30$ and for all prime $n>13$.
\end{theorem}
It will be necessary to use Lehmer pairs to prove for certain values of $r$ that $(\ref{eq:main})$ has no solutions.

\section{Descent to eight cases}\label{casen2}
Equation~\eqref{eq:main} can be rewritten as:
$
5x(x^2+6r^2) = y^p.
$
Letting $y = 5w$, we can rewrite as:
\begin{equation}\label{eq:mainfact}
x(x^2+6r^2) = 5^{p-1}w^p.
\end{equation}
Note that $\gcd(x, x^2 + 6r^2) \in \{1,2,3,6\}$ depending on whether $2,3$ divides $x$ or not. Thus,
we are now able to divide into eight cases.

\begin{adjustbox}{width =\textwidth}
\begin{tabular}{|c|c|c|c|}
\hline
{\bf Case} & {\bf Conditions on $x$} & {\bf Descent equations} & {\bf Equation of signature $(p,p,2)$}\\
\hline\hline
$1$ &  $5 \mid x$ and $6 \nmid x$ &
\begin{tabular}{@{}c@{}}$x=5^{p-1}w_1^p$ \\ $x^2+6r^2=w_2^p$ \end{tabular} & $w_2^p-5^{2p-2}w_1^{2p}=6r^2$ 
\\

\hline 
$2$ &  $5 \mid x$ and $2 \mid x$ and $3\nmid x$ &
\begin{tabular}{@{}c@{}}$x=2^{p-1}5^{p-1}w_1^p$ \\ $x^2+6r^2=2w_2^p$ \end{tabular}  & 
$w_2^p-2^{2p-3}5^{2p-2} w_1^{2p}=3r^2$ \\

\hline 
$3$ &  $5 \mid x$ and $3 \mid x$ and $2\nmid x$&
\begin{tabular}{@{}c@{}}$x=3^{p-1}5^{p-1}w_1^p$ \\ $x^2+6r^2=3w_2^p$ \end{tabular}  & 
$w_2^p-3^{2p-3}5^{2p-2} w_1^{2p}=2r^2$\\

\hline 
$4$ &  $5 \mid x$ and $6 \mid x$ &
\begin{tabular}{@{}c@{}}$x=6^{p-1}5^{p-1}w_1^p$ \\ $x^2+6r^2=6w_2^p$ \end{tabular}  & 
$w_2^p-6^{2p-3}5^{2p-2}w_1^{2p}=r^2$\\

\hline
$5$ &  $5 \nmid x$ and $6 \nmid x$ &
\begin{tabular}{@{}c@{}}$x=w_1^p$ \\ $x^2+6r^2=5^{p-1}w_2^p$ \end{tabular}  & 
$5^{p-1}w_2^p-w_1^{2p}=6r^2$\\

\hline 
$6$ &  $5 \nmid x$ and $2 \mid x$ and $3\nmid x$ &
\begin{tabular}{@{}c@{}}$x=2^{p-1}w_1^p$ \\ $x^2+6r^2=2\cdot5^{p-1}w_2^p$ \end{tabular} & 
$5^{p-1} w_2^p-2^{2p-3} w_1^{2p}=3r^2$ \\

\hline 
$7$ &  $5 \nmid x$ and $3 \mid x$ and $2\nmid x$&
\begin{tabular}{@{}c@{}}$x=3^{p-1}w_1^p$ \\ $x^2+6r^2=3\cdot5^{p-1}w_2^p$ \end{tabular} & 
$5^{p-1}w_2^p-3^{2p-3} w_1^{2p}=2r^2$\\

\hline 
$8$ &  $5 \nmid x$ and $6 \mid x$ &
\begin{tabular}{@{}c@{}}$x=6^{p-1}w_1^p$ \\ $x^2+6r^2=6\cdot5^{p-1}w_2^p$ \end{tabular} & 
$5^{p-1} w_2^p-6^{2p-3}w_1^{2p}=r^2$\\

\hline 
\end{tabular}
\label{table:Cases}
\end{adjustbox}

\section{Tables of computations}\label{sec:tablecomps}

In this section, we apply the Theorems and Lemmas of Section~\ref{sec:background} in order to fully resolve equation~\eqref{eq:main} and prove Theorem~\ref{thm:principal01}.

\subsection{Case 1}

We first apply Theorem $\ref{thm:mignotte}$ to obtain the bound $p\leq 34365$ when $|r|\leq 1.7\times 10^{2486}$. Thus, when we focus on $5\leq p \leq 34365$ and $1\leq r\leq 10^6$ we have the following table containing the obtained information after computational calculations.
~

\[
\begin{adjustbox}{width =\textwidth}
\begin{tabular}{|c|c|c|c|c|}
\hline
\text{Exponent }p & \makecell{Number of eqns\\ surviving \\Lemma \ref{lem:Sophiecriterion}}& \makecell{Number of eqns \\surviving local \\solubility tests} & \makecell{Number of eqns \\surviving \\further descent} & \makecell{Thue eqns \\ not solved  \\ by Magma} \\
\hline
5 & 258519 & 33667 & 52 & 0\\
\hline
7 & 102711 & 46794 & 5 & 0\\
\hline
11 & 2690 & 1364 & 1 & 0\\
\hline
13 & 5855 & 3044 & 1 & 0 \\
\hline
17 & 752 & 415 & 0 & 0 \\
\hline
19 & 1644 & 858 & 0 & 0\\
\hline
23 & 10 & 6 & 0 & 0\\
\hline
29 & 23 & 9 & 0 &0 \\
\hline
31 & 61 & 32 & 0 & 0\\
\hline
37 & 1 & 1 & 0 & 0\\
\hline
41 & 2 & 1 & 0 & 0\\
\hline
$43 \leq p \leq 34365$ & 0 &0 &0 &0\\
\hline
\end{tabular}
\end{adjustbox}
\]

\subsection{Case 2}

For $p=5$ we have $w_2^5-2^7\cdot5^8w_1^{10}=3r^2$. Letting $X=w_2/w_1^2$ and $Y=3r/w_1^5$ we obtain the hyperelliptic curve 
\[
Y^2=3X^5-3\cdot2^7\cdot 5^8
\]
whose Jacobian has rank 1. We can show, using Magma, that $C(\Q)=\{\infty\}$ which implies the curve only has trivial solutions.

Using Theorem $\ref{thm:mignotte}$ we bound $p\leq 56565$ for $|r|\leq 6.8\times 10^{4092}$. Thus, when we focus on $7\leq p \leq 56565$ and $1\leq r\leq 10^6$ we have the following table.
\[
\begin{adjustbox}{width =\textwidth}
\begin{tabular}{|c|c|c|c|c|}
\hline
\text{Exponent }p & \makecell{Number of eqns\\ surviving \\Lemma \ref{lem:Sophiecriterion}}& \makecell{Number of eqns \\surviving local \\solubility tests} & \makecell{Number of eqns \\surviving \\further descent} & \makecell{Thue eqns \\ not solved  \\ by Magma}\\
\hline
7 & 60261 & 21654 & 2 & 0\\
\hline
11 & 895 & 347 & 1 & 0\\
\hline
13 & 783 & 401 & 1 & 0\\
\hline
17 & 126 & 69 & 1 & 1\\
\hline
19 & 656 & 296 & 1 & 1\\
\hline
23 & 3 & 1 & 1 & 1\\
\hline
31 & 4 & 3 & 0 & 0\\
\hline
$37\leq p\leq 56565$ & 0 & 0 & 0 & 0\\
\hline
\end{tabular}
\end{adjustbox}
\]
The Thue equations that could not be solved by Magma, were solved by PARI/GP
\[
\begin{tabular}{|c|c|c|c|}
\hline
$p$&$r$& \textit{Thue equation} & \textit{Solution}\\
\hline
17 & $3^8$ & $w_2^{17}-2^{31}\cdot5^{32}w_1^{34}=3^{17}$ & $(3,0)$ \\
\hline
19 & $3^9$ & $w_2^{19}-2^{35}\cdot5^{36}w_1^{38}=3^{19}$ & $(3,0)$ \\
\hline
23 & $3^{11}$ & $w_2^{23}-2^{43}\cdot5^{44}w_1^{46}=3^{23}$ & $(3,0)$\\
\hline
\end{tabular}
\]
using the commands \textit{thueinit} and \textit{thue}. Observe that the solutions found contradict the condition $w_1\cdot w_2\neq 0$. Moreover, observe that all the equations that were solved are of signature $(p,2p,p)$.

\subsection{Case 3}

For $p=5$ we have $w_2^5-3^7\cdot 5^8w_1^{10}=2r^2$. Choosing the change of variable $X=w_2/w_1^2$ and $Y=2r/w_1^5$ we obtain the hyperelliptic curve 
\[
Y^2=2X^5-2\cdot3^7\cdot 5^8
\]
whose Jacobian has rank 1. We can show, using Magma, that $C(\Q)=\{\infty\}$ which implies the curve only has trivial solutions.

Using Theorem $\ref{thm:mignotte}$ we bound $p\leq 69551$ for $|r|\leq 3.8\times 10^{5881}$. Thus, when we focus on $7\leq p \leq 69551$ and $1\leq r\leq 10^6$ we have the following table.
\[
\begin{adjustbox}{width =\textwidth}
\begin{tabular}{|c|c|c|c|c|}
\hline
\text{Exponent }p & \makecell{Number of eqns\\ surviving \\Lemma \ref{lem:Sophiecriterion}}& \makecell{Number of eqns \\surviving local \\solubility tests} & \makecell{Number of eqns \\surviving \\further descent} & \makecell{Thue eqns \\ not solved  \\ by Magma} \\
\hline
7  & 18911 & 8045 & 3 & 0 \\ \hline
11 & 1639  & 690  & 2 & 0 \\ \hline
13 & 4059  & 3314 & 2 & 0 \\ \hline
17 & 137   & 76   & 1 & 1 \\ \hline
19 & 271   & 141  & 1 & 1 \\ \hline
23 & 11    & 5    & 1 & 1 \\ \hline
29 & 5     & 2    & 1 & 1 \\ \hline
31 & 9     & 6    & 1 & 1 \\ \hline
37 & 1     & 1    & 1 & 1 \\ \hline
$41\leq p\leq 69551$ & 0 & 0 & 0 & 0 \\ \hline
\end{tabular}
\end{adjustbox}
\]

The Thue equations that could not be solved by Magma, were solved by PARI/GP
\[
\begin{tabular}{|c|c|c|c|}
\hline
$p$&$r$& \textit{Thue equation} & \textit{Solution}\\
\hline
17 & $2^8$ & $w_2^{17}-3^{31}\cdot5^{32}w_1^{34}=2^{17}$ & $(2,0)$ \\
\hline
19 & $2^9$ & $w_2^{19}-3^{35}\cdot5^{36}w_1^{38}=2^{19}$ & $(2,0)$ \\
\hline
23 & $2^{11}$ & $w_2^{23}-3^{43}\cdot5^{44}w_1^{38}=2^{23}$ & $(2,0)$ \\
\hline
29 & $2^{14}$ & $w_2^{29}-3^{55}\cdot5^{56}w_1^{38}=2^{29}$ & $(2,0)$ \\
\hline
31 & $2^{15}$ & $w_2^{31}-3^{59}\cdot5^{60}w_1^{38}=2^{31}$ & $(2,0)$ \\
\hline
37 & $2^{18}$ & $w_2^{37}-3^{71}\cdot5^{72}w_1^{38}=2^{37}$ & $(2,0)$ \\
\hline
\end{tabular}
\]
using the commands \textit{thueinit} and \textit{thue}. Observe that the solutions found contradict the condition $w_1\cdot w_2\neq 0$. Moreover, observe that all the equations that were solved are of signature $(p,2p,p)$.

\subsection{Case 4}

For $p=5$ we have $w_2^5-6^7\cdot5^8w_1^{10}=r^2$. Choosing the change of variable $X=w_2/w_1^2$ and $Y=r/w_1^5$ we obtain the hyperelliptic curve 
\[
Y^2=X^5-6^7\cdot 5^8
\]
whose Jacobian has rank 1. We can show, using Magma, that $C(\Q)=\{\infty\}$ which implies the curve only has trivial solutions.

Using Theorem $\ref{thm:mignotte}$ we bound $p\leq 91751$ for $|r|\leq 1.9\times 10^{6639}$. Thus, when we focus on $7\leq p \leq 91751$ and $2\leq r\leq 10^6$ we have the following table.

\[
\begin{adjustbox}{width =\textwidth}
\centering
\begin{tabular}{|c|c|c|c|c|}
\hline
\text{Exponent }p & \makecell{Number of eqns\\ surviving \\Lemma \ref{lem:Sophiecriterion}}& \makecell{Number of eqns \\surviving local \\solubility tests} & \makecell{Number of eqns \\surviving \\further descent} & \makecell{Thue eqns \\ not solved  \\ by Magma} \\
\hline
7  & 22887 & 12311 & 1 & 0\\
\hline
11 & 1469  & 588 & 0 & 0 \\
\hline
13 & 237 & 97 & 0 & 0 \\
\hline
17 & 205 & 118 & 0 & 0 \\
\hline
19 & 2480 & 1273 & 0 & 0 \\
\hline
23 & 2 & 2 & 0 & 0 \\
\hline
29 & 2 & 2 & 0 & 0 \\
\hline
31 & 9 & 4 & 0 & 0 \\
\hline
37 & 2 & 1 & 0 & 0 \\
\hline
$41\leq 91751$ & 0 & 0 & 0 & 0 \\
\hline
\end{tabular}
\end{adjustbox}
\]

We noticed a remarkable behaviour when $r=1$ and $7\leq p\leq 91751$.
 
 We observed that application of Lemma \ref{lem:Sophiecriterion} and the local solubility tests were always failing for every $p$. 
 Due to this situation, we decided to use \textit{Lehmer pairs} to solve it.

Let $K=\Q(\sqrt{-6})$ and write $\mathcal{O}_{K}$ for its ring of integers. This has class group isomorphic to $\Z/2\Z$. We consider the equation $x^2+6=6w_2^p$ where $x=30^{p-1}w_1^p$ and $w_1$, $w_2$ as before. Take 
\begin{align*}
6w_2^p&=x^2+6\\
&=(x+\sqrt{-6})(x-\sqrt{-6}).
\end{align*}
It follows that  
\begin{align*}
(x+\sqrt{-6})\mathcal{O}_K=\p_2\p_3\zeta^p
\end{align*}
where $\p_2=(2,\sqrt{-6})$, $\p_3=(3,\sqrt{-6})$ are the primes above $2$, $3$ and $\zeta$ is an ideal in $\mathcal{O}_K$. We write
\begin{align*}
(x+\sqrt{-6})\mathcal{O}_K=2^{(1-p)/2}3^{(1-p)/2}(\p_2\p_3\zeta)^p
\end{align*}
and $\q=\p_2\p_3$. It follows that $\q\zeta=(\gamma)\in\OO_K$ is principal and $\gamma=u+v\sqrt{-6}\in\OO_K$ with $u,v\in\Z$. After a possible change of sign we obtain 
\[
x+\sqrt{-6}=\dfrac{\gamma^p}{6^{(p-1)/2}}
\]
Subtracting and conjugating we obtain 
\begin{equation}\label{eq:subconj}
\dfrac{\gamma^p}{6^{(p-1)/2}}-\dfrac{\overline{\gamma}^p}{6^{(p-1)/2}}=2\sqrt{-6}
\end{equation}
or equivalently
\[\dfrac{\gamma^p}{6^{p/2}}-\dfrac{\overline{\gamma}^p}{6^{p/2}}=2i.
\]
Let $L=\Q(\sqrt{-1},\sqrt{6})$. Taking $\alpha=\gamma/\sqrt{6}$ and $\beta=\overline{\gamma}/\sqrt{6}$. 

\begin{lemma}\label{lemma:van01}
Let $\alpha$ and $\beta$ as above. Then $\alpha$ and $\beta$ are algebraic integers. Moreover $(\alpha+\beta)^2$ and $\alpha\beta$ are nonzero coprime rational integers and $\alpha/\beta$ is not a unit.
\end{lemma} 
\begin{proof}
Let $\q\OO_L=\sqrt{6}\OO_L$. By definition $\q|\gamma,\overline{\gamma}$ which implies that $\alpha$ and $\beta$ are algebraic integers. Now we compute $(\alpha+\beta)^2=4u^2/3$.

Since $\p_3|\sqrt{-6}$ we conclude that $\p_3|u$ and so $3|u$. So $(\alpha+\beta)^2$ is a rational integer, i.e., $(\alpha+\beta)^2\in\Z$. If $(\alpha+\beta)^2=0$ then $u=0$, however this will imply that $30^pw_1^p=0$ contradicting $w_1\cdot w_2\neq 0$. Thus $(\alpha+\beta)^2$ is a nonzero ration integer. Moreover $\alpha\beta=\gamma\overline{\gamma}/6$ is a nonzero rational integer since $3|u$ and $\p_2|\gamma,\overline{\gamma}$.

We now check that $(\alpha+\beta)^2$ and $\alpha\beta$ are coprime. Suppose they are not coprime. Then there exists a prime $\p$ of $\OO_L$ which divides both. Then $\p$ divides $\alpha$ and $\beta$ and from the equation above $\p$ divides $(w_2)\OO_L$ and $2\sqrt{-6}$. We contradicted our assumption  of $(w_1,w_2)$ being a nontrivial coprime solution.

Finally, $\alpha/\beta=\gamma/\overline{\gamma}\in\OO_K$ is not a root of unity because the only roots of unity in $K$ are $\pm1$ which implie $\gamma=\pm\overline{\gamma}$ implying either $u=0$ or $v=0$ which is a contradiction.
\end{proof}
From Lemma \ref{lemma:van01} we have that the pair $(\alpha,\beta)$ is Lehmer pair and we denote by $\tilde{u}_k$ the associate Lehmer sequence. Substituting we see that
\[
\left(\dfrac{\alpha-\beta}{2i}\right)\left(\dfrac{\alpha^p-\beta^p}{\alpha-\beta}\right)=1.
\]
Hence, we get 
\[
\left(\dfrac{\alpha^p-\beta^p}{\alpha-\beta}\right)=1,
\]
thus $v = \pm 1$. By theorem~\ref{BHV}, we immediately deduce that $p \in \{5,7,11,13\}$. For a given prime $p$, using equation~\eqref{eq:subconj}, we see that $u$ is a root of the polynomial:
\[
\frac{1}{2\cdot\sqrt{-6}\cdot\sqrt{6}^{(p-1)/2}}((u+v\sqrt{-6})^p - (u-v\sqrt{-6})^p) -1.
\]

Computing the roots of these polynomials, we find that there are no solutions. 

\subsection{Case 5}

For $p=5$ we have $5^4w_2^5-w_1^{10}=6r^2$. Choosing the change of variable $X=w_2/w_1^2$ and $Y=6r/w_1^5$ we obtain the hyperelliptic curve 
\[
Y^2=6\cdot5^4X^5-6
\]
whose Jacobian has rank 0. We can show, using Magma, that $C(\Q)=\{\infty\}$ which implies the curve only has trivial solutions.

Using Theorem $\ref{thm:mignotte}$ we bound $p\leq 17183$ for $|r|\leq 8.3\times 10^{1242}$. Thus, when we focus on $7\leq p \leq 17183$ and $1\leq r\leq 10^6$ we have the following table.
\[
\begin{adjustbox}{width =\textwidth}
\begin{tabular}{|c|c|c|c|c|}
\hline
\text{Exponent }p & \makecell{Number of eqns\\ surviving \\Lemma \ref{lem:Sophiecriterion}}& \makecell{Number of eqns \\surviving local \\solubility tests} & \makecell{Number of eqns \\surviving \\further descent} &\makecell{Thue eqns \\ not solved  \\ by Magma} \\
\hline
7 & 99944& 55754 & 71 & 0\\
\hline
11 & 3345 & 1417 & 0 & 0\\
\hline
13 & 871 & 446 & 0 & 0\\
\hline
17 & 1042 & 552 & 0 & 0\\
\hline
19 & 892 & 487 & 0 & 0\\
\hline
23 & 20 & 9 & 0 & 0 \\
\hline
29 & 2 & 2 & 0 & 0 \\
\hline  
31 & 23 & 11 &0 & 0\\
\hline
37 & 1 & 1 & 0 & 0\\
\hline
$41\leq p\leq 17183$& 0 &0 &0 & 0\\
\hline
\end{tabular}
\end{adjustbox}
\]

\subsection{Case 6}

Using Theorem $\ref{thm:mignotte}$ we bound $p\leq 9101$ for $|r|\leq 9.4\times 10^{657}$. Thus, when we focus on $5\leq p \leq 9101$ and $1\leq r\leq 10^6$ we have the following table.
\[
\begin{adjustbox}{width =\textwidth}
\begin{tabular}{|c|c|c|c|c|}
\hline
\text{Exponent }p & \makecell{Number of eqns\\ surviving \\Lemma \ref{lem:Sophiecriterion}}& \makecell{Number of eqns \\surviving local \\solubility tests} & \makecell{Number of eqns \\surviving \\further descent} & \makecell{Thue eqns \\ not solved  \\ by Magma} \\
\hline
5 & 94300 & 10723 & 35 &0\\
\hline
7 & 28798 & 22060 & 0 & 0\\
\hline
11 & 757 & 388 & 0 & 0\\
\hline
13 & 7923 & 4238 & 0 & 0\\
\hline 
17 & 95 & 47 & 0 & 0\\
\hline
19 & 686 & 310 & 0 & 0\\
\hline
23 & 1 & 1 & 0 & 0\\
\hline
29 & 5 & 4 & 0 & 0\\
\hline
31 & 14 & 7 & 0 & 0\\
\hline
37 & 2 & 2 & 0 & 0\\
\hline
$41\leq p\leq 9101$& 0 &0 &0 & 0\\
\hline
\end{tabular}
\end{adjustbox}
\]

\subsection{Case 7}

Using Theorem $\ref{thm:mignotte}$ we bound $p\leq 22515$ for $|r|\leq 5.4\times 10^{1628}$. Thus, when we focus on $5\leq p \leq 22515$ and $1\leq r\leq 10^6$ we have the following table.
\[
\begin{adjustbox}{width =\textwidth}
\begin{tabular}{|c|c|c|c|c|}
\hline
\text{Exponent }p & \makecell{Number of eqns\\ surviving \\Lemma \ref{lem:Sophiecriterion}}& \makecell{Number of eqns \\surviving local \\solubility tests} & \makecell{Number of eqns \\surviving \\further descent} & \makecell{Thue eqns \\ not solved  \\ by Magma} \\
\hline
5 & 36897 & 5233 & 94 &0\\
\hline
7 & 26109 & 12665 & 16 & 0\\
\hline
11 & 4079 & 1637 & 0 & 0 \\
\hline
13 & 1649 & 854 & 0 & 0 \\
\hline
17 & 945 & 606 & 0 & 0 \\
\hline
19 & 686 & 459 & 0 & 0 \\
\hline
23 & 12 & 5 & 0 & 0 \\
\hline
29 & 5 & 3 & 0 & 0 \\
\hline
31 & 35 & 21 & 0 & 0 \\
\hline
37 & 1 & 0 & 0 & 0 \\
\hline
41 & 1 & 0 & 0 & 0 \\ 
\hline
$43\leq p \leq 22515$ & 0 & 0 & 0 & 0 \\
\hline
\end{tabular}
\end{adjustbox}
\]

\subsection{Case 8}

For $p=5$ we have $5^4w_2^5-6^7w_1^{10}=r^2$. Choosing the change of variable $X=w_2/w_1^2$ and $Y=r/w_1^5$ we obtain the hyperelliptic curve 
\[
Y^2=5^4X^5-6^7
\]
whose Jacobian has rank 0. We can show, using Magma, that $C(\Q)=\{\infty\}$ which implies the curve only has trivial solutions.

Using Theorem $\ref{thm:mignotte}$ we bound $p\leq 44855$ for $|r|\leq 2.8\times 10^{3245}$. Thus, when we focus on $7\leq p \leq 44855$ and $1\leq r\leq 10^6$ we have the following table.
\[
\begin{adjustbox}{width =\textwidth}
\begin{tabular}{|c|c|c|c|c|}
\hline
\text{Exponent }p & \makecell{Number of eqns\\ surviving \\Lemma \ref{lem:Sophiecriterion}}& \makecell{Number of eqns \\surviving local \\solubility tests} & \makecell{Number of eqns \\surviving \\further descent} & \makecell{Thue eqns \\ not solved  \\ by Magma} \\
\hline
7 & 18672 & 6518 & 0 & 0\\
\hline
11 & 904 & 381 & 0 & 0\\
\hline
13 & 561 & 225 & 0 & 0\\
\hline
17 & 122 & 89 & 0 & 0\\
\hline
19 & 620 & 375 & 0 & 0\\
\hline
23 & 2 & 1 & 0 & 0\\
\hline
29 & 1 & 0 & 0 & 0\\
\hline
31 & 20 & 11 & 0 & 0\\
\hline
41 & 1 & 0 & 0 & 0\\
\hline
$43\leq p \leq 44855$ & 0 & 0 & 0 & 0\\
\hline
\end{tabular}
\end{adjustbox}
\]

This completes the proof of Theorem~\ref{thm:principal01}.

\section*{Acknowledgements}
\thispagestyle{empty}
I would like to thank CIMAT Unidad M\'erida, where  I worked under a postdoctoral fellowship provided by CONACYT grant FOMIX-YUC 221183. I would also like to thank Professor John E. Cremona and The University of Warwick for allowing me access to the Number Theory Group computers where the computations were done.

\bibliographystyle{plain}

\end{document}